\documentclass[12pt]{amsart}
\usepackage{amsmath,amssymb,amsbsy,amsfonts,amsthm,latexsym,mathabx,
            amsopn,amstext,amsxtra,euscript,amscd,
            stmaryrd,
            mathrsfs,
            cite,array,mathtools,enumerate,bm}
\usepackage{enumitem}
\usepackage{float} 
\usepackage[english]{babel}
\usepackage{mathtools}
\usepackage{todonotes}
\usepackage{url}
\usepackage[colorlinks,linkcolor=blue,anchorcolor=blue,citecolor=blue,backref=page]{hyperref}
\def\le{\leqslant}
\def\ge{\geqslant}

\usepackage{color}

\usepackage{mathtools}
\usepackage{todonotes}
\usepackage[norefs,nocites]{refcheck}

\def \balpha{\bm{\alpha}}
\def \bbeta{\bm{\beta}}

\hypersetup{breaklinks=true}

\usepackage[english]{babel}

\usepackage{mathtools}
\usepackage{todonotes}

\usepackage{enumitem}

\begin{document}

\newtheorem{theorem}{Theorem}
\newtheorem{lemma}[theorem]{Lemma}
\newtheorem{claim}[theorem]{Claim}
\newtheorem{cor}[theorem]{Corollary}
\newtheorem{prop}[theorem]{Proposition}
\newtheorem{definition}{Definition}
\newtheorem{question}[theorem]{Open Question}
\newtheorem{example}[theorem]{Example}
\newtheorem{remark}[theorem]{Remark}

\numberwithin{equation}{section}
\numberwithin{theorem}{section}

 \newcommand{\F}{\mathbb{F}}
\newcommand{\K}{\mathbb{K}}
\newcommand{\D}[1]{D\(#1\)}
\def\scr{\scriptstyle}
\def\\{\cr}
\def\({\left(}
\def\){\right)}
\def\[{\left[}
\def\]{\right]}
\def\<{\langle}
\def\>{\rangle}
\def\fl#1{\left\lfloor#1\right\rfloor}
\def\rf#1{\left\lceil#1\right\rceil}
\def\le{\leqslant}
\def\ge{\geqslant}
\def\eps{\varepsilon}
\def\mand{\qquad\mbox{and}\qquad}

\def\Res{\mathrm{Res}}
\def\vec#1{\mathbf{#1}}

\def \vs {\vec{s}}

\def\bl#1{\begin{color}{blue}#1\end{color}} 

\newcommand{\red}[1]{{\color{red}#1}}

\newcommand{\Fq}{\mathbb{F}_q}
\newcommand{\Fp}{\mathbb{F}_p}
\newcommand{\Disc}[1]{\mathrm{Disc}\(#1\)}

\newcommand{\Z}{\mathbb{Z}}
\renewcommand{\L}{\mathbb{L}}

\def\cA{{\mathcal A}}
\def\cB{{\mathcal B}}
\def\cC{{\mathcal C}}
\def\cD{{\mathcal D}}
\def\cE{{\mathcal E}}
\def\cF{{\mathcal F}}
\def\cG{{\mathcal G}}
\def\cH{{\mathcal H}}
\def\cI{{\mathcal I}}
\def\cJ{{\mathcal J}}
\def\cK{{\mathcal K}}
\def\cL{{\mathcal L}}
\def\cM{{\mathcal M}}
\def\cN{{\mathcal N}}
\def\cO{{\mathcal O}}
\def\cP{{\mathcal P}}
\def\cQ{{\mathcal Q}}
\def\cR{{\mathcal R}}
\def\cS{{\mathcal S}}
\def\cT{{\mathcal T}}
\def\cU{{\mathcal U}}
\def\cV{{\mathcal V}}
\def\cW{{\mathcal W}}
\def\cX{{\mathcal X}}
\def\cY{{\mathcal Y}}
\def\cZ{{\mathcal Z}}

\def \brho{\boldsymbol{\rho}}

\def \pf {\mathfrak p}

\def \Prob{{\mathrm {}}}
\def\e{\mathbf{e}}
\def\ep{{\mathbf{\,e}}_p}
\def\epp{{\mathbf{\,e}}_{p^2}}
\def\em{{\mathbf{\,e}}_m}

\def \F{{\mathbb F}}
\def \K{{\mathbb K}}
\def \Z{{\mathbb Z}}
\def \N{{\mathbb N}}
\def \Q{{\mathbb Q}}
\def \R{{\mathbb R}}

\def\GL{\operatorname{GL}}
\def\SL{\operatorname{SL}}
\def\PGL{\operatorname{PGL}}
\def\PSL{\operatorname{PSL}}

\def\ep{\mathbf{e}_p}
\def\eq{\mathbf{e}_q}
\def \SNQ{{\mathcal S}(N,Q)}

\def\Mob{M{\"o}bius }


\title[M\"obius Orbit at Smooth Times]{On the Dynamical System Generated by the M{\"o}bius  Transformation  at 
Smooth Times}

 \author[L. M{\'e}rai]{L{\'a}szl{\'o} M{\'e}rai}
\address{L.M.:  Department of Computer Algebra, E\"otv\"os Lor\'and University,  H-1117 Budapest, Pazm\'any P\'eter s\'et\'any 1/C, Hungary}
\email{merai@inf.elte.hu}

\author[I.~E.~Shparlinski]{Igor E. Shparlinski}
\address{I.E.S.: School of Mathematics and Statistics, University of New South Wales.
Sydney, NSW 2052, Australia}
\email{igor.shparlinski@unsw.edu.au}

\begin{abstract}
We study the distribution of the sequence of the 
first $N$ elements of the discrete 
dynamical system 
generated by the 
\Mob transformation   $x \mapsto (\alpha x + \beta)/(\gamma  x + \delta)$ 
 over a finite 
field of $p$ elements at the moments of time that correspond to 
$Q$-smooth numbers, that is, to numbers composed out 
of primes up to $Q$. 
In particular, we obtain nontrivial estimates of exponential sums with 
such sequences. 
\end{abstract}
 

\maketitle

%

\section{Introduction}

\subsection{Motivation} 
Let $p$ be a sufficiently large prime and let $\F_p$ be the field of $p$ elements
which we identify with the least residue system modulo $p$, that is, with the set
$\{0, \ldots, p-1\}$.

With any  nonsingular matrix
\begin{equation}
\label{eq:MatrM}
A = \begin{pmatrix} \alpha & \beta\\ \gamma & \delta \end{pmatrix} \in \GL_2(\F_p), 
\end{equation}
we consider the \Mob transformation  
$x \mapsto \psi(x)$ associated with $A$  where 
\begin{equation}
\label{eq:Perm}
 \psi(x)  =  \frac{\alpha x + \beta}{\gamma  x + \delta}. 
\end{equation}
Throughout the paper we always assume that 
\begin{equation}
\label{eq:c<>0}
\gamma \ne 0.
\end{equation}

Investigating the distributional properties of elements in orbits  
of the  discrete dynamical system generated by iterations of $\psi$ 
or some other polynomial or rational functions  over  finite fields and residue rings, has been a very active area of research, especially in the theory of pseudorandom number 
generators, see~\cite{BhShp, GuNiSh, MeShp20, MeShp21, MeShp24, NiSh1,NiSh2, NiWi, OstShp3, Shp} and references therein. In fact, in the theory of pseudorandom 
number generators, typically only the special case  $\psi(x) = \alpha x^{-1}  + \beta$ is considered (which is 
computationally more efficient). 
Moreover, the sequences  generated by   
iterations of  any map of the form~\eqref{eq:Perm}, as in~\eqref{eq:Gen} below, 
 can be reduced to sequences produced by this special map via 
a linear transformation, which typically does not affect their distributional and other important 
properties. Here however we prefer to consider the \Mob transformation  in the traditional  form~\eqref{eq:Perm}.


Here we are interested in more arithmetic aspects of this problem 
where one studes the distribution of elements in orbits of the \Mob transformation 
at the moments of time that correspond to number theoretically interesting sequences. 
For example, in~\cite{MeShp21} the orbits are studied at the prime moments 
of time. In this work we concentrate on smooth times for a rather high level of smoothness, 
which  looks like a harder question since the  sequences
of very smooth integers, which we consider,  are much sparse than primes. 

\subsection{Formal set-up} 
More precisely, let $u_0, u_1, \ldots$ be an orbit of the dynamical system 
generated by $\psi$ that originates at some $u_0 \in \F_p$, 
that is, 
\begin{equation}
\label{eq:Gen}
u_{n} =\psi\(u_{n-1}\),
 \qquad n =  1,2,
\ldots\,,
\end{equation}
where $u_0$ is the {\it initial value\/}, with the convention $\psi(-\delta/\gamma)=\alpha/\gamma$ which is well defined 
under the assumption~\eqref{eq:c<>0}. 

We can also write 
$$
u_n = \psi^{n}(u_0), \qquad n =  1,2,
\ldots\,,
$$
where $\psi^{0}$ is the identity map and $ \psi^{n}$
is the $n$th composition of $\psi$.

Since for any $A \in \GL_2(\F_p)$ the \Mob transformation~\eqref{eq:Perm} is reversible, 
it  is obvious that the sequence~\eqref{eq:Gen} is purely
periodic with some period $t \le p$, see~\cite{Chou,FN}
for several results about the possible values of~$t$.
 For example, it is known when such sequences
achieve the
largest possible period, which is obviously $t = p$, see~\cite{FN}.

The series of works~\cite{GuNiSh,NiSh1,NiSh2}  
 is devoted to the special case of the transformation  $\psi(x) = \alpha x^{-1}  +  \beta$ where several results 
about the distribution of elements of the sequence~\eqref{eq:Gen}
are given. Quite naturally, these results are based on bounds of exponential sums such as
\begin{equation}
\label{eq:Sing Sum}
S_{h}(N) =  \sum_{n=1}^{N} \ep\(hu_n\),
\end{equation}
where for an integer $q$ and a complex $z$ we define
$$
\eq(z) = \exp(2 \pi i z/q).
$$  
We also remark that a version of~\cite[Lemma~5.3]{AbSh} improves and generalises the bounds of~\cite{NiSh1} 
on $S_{h}(N)$, see Lemma~\ref{lem:SingSum-s} below for further generalisation which stems from~\cite[Lemma~6]{MeShp21}. It can easily be extended to  multidimensional
settings~\cite{GuNiSh} and thus has direct applications to the theory of pseudorandom number generators. 

In~\cite{MeShp21}, we have investigated  elements in orbits of the \Mob transformation at prime times and, in particular, have shown that for
$$
T_h(N) =  \sum_{\substack{\ell \le N\\\ell~\mathrm{prime}}} \ep\left(hu_\ell\right)
$$
we have 
$$
|T_h(N)| \leq Np^{-\eta} 
$$
if the period $t\geq p^{3/4+\varepsilon}$ and $p^B\leq N\leq p^C$ for some positive real numbers $\varepsilon, B,C$, where $\eta>0$ may depend on these parameters and $p$ is sufficiently large. 

In this paper,  we study the distribution of trajectories of the \Mob transformation
at the moments of time that correspond to  {\it $Q$-smooth\/} numbers,
where as usual, we say that an integer $n$ is $Q$-smooth if
the largest  prime divisor $P(n)$ of $n$ satisfies $P(n) \le Q$. Let $\SNQ$ be the set of $Q$-smooth numbers up to $N$,
$$
\SNQ = \{n \le N~:~ n\ \text{is $Q$-smooth}\}.
$$
We recall, that for the number $\Psi(N,Q)=\#\SNQ $ of $Q$-smooth numbers up to $N$ we have
$$
  \Psi(N,Q)=  N\rho(u)\left(1+O\left(\frac{\log(u+1)}{\log Q}\right)\right),
$$
  where, as usual,
$$
 u=\frac{\log N}{\log Q}, 
$$
and $\rho(u)$ is the so-called \textit{Dickman's} functions satisfying
$$
\rho(u)=\left(\frac{e+o(1)}{u\log u}\right)^u \quad \text{as } u\rightarrow \infty,
$$ 
in the range
$$
Q>\exp\left((\log \log N)^{5/3+\varepsilon}\right),
$$
or, alternatively, $1\leq u\leq \exp\left((\log Q)^{3/5-\varepsilon}\right)$, with any fixed $\varepsilon > 0$, 
see~\cite{Ten} for more details.

It is also useful to recall that if $Q = (\log N)^{A+o(1)}$ for some constant $A>1$ then 
\begin{equation}\label{eq:psi-small Q}
  \Psi(N,Q)=  N^{1-1/A + o(1)}, 
  \end{equation}
  see, for instance,~\cite[Equation~(1.14)]{Gran}, 

Our goal is to investigate the exponential sum
$$
T_h(N,Q) =  \sum_{n \in \SNQ} \ep\(hu_n\).
$$

We also note that the results of~\cite{BFGS} can be considered 
as results on the behaviour at smooth moments of time of the 
dynamical system generated by the linear transformation 
$w \mapsto gw$ on $\F_p$, that is, of the sequence $u_0g^s$, 
where $n$ runs through the set $\SNQ$.

\subsection{Our results}

We establish the following upper bound on the
sums $T_h(N, Q)$, which is nontrivial in a wide range of parameters.


\begin{theorem}
\label{thm:SmoothSum3}
For any  $\varepsilon>0$ and $B \ge 1$, there exists a $\delta>0$ with the following property. Assume that the period $t$ of the sequence~\eqref{eq:Gen} satisfies 
$t  \ge Q p^{1/2+\varepsilon}$. Assume also that 
$p^B\geq N \geq Q^2 p^{1/2+ \varepsilon }$. 
Then for
 any $h \in \Fp^*$, we have
$$
T_h(N,Q) \leq c  N^{1-\delta}Q 
$$
where $c$  and $\delta$ may depend only on  $B$ and $\varepsilon$. 
\end{theorem}
We remark that one can choose any fixed $\delta > \varepsilon/(8B)$  in Theorem~\ref{thm:SmoothSum3}. 
Furthermore, we see from  the bound~\eqref{eq:psi-small Q} that there are 
$\eta> 0$, $\kappa> $ and $A>1$, depending only on  $B$ and $\varepsilon$, such that 
under the conditions of Theorem~\ref{thm:SmoothSum3} we have 
$$
T_h(N,Q) \leq c \Psi(N,Q)^{1-\eta}
$$
provided
$$
N^\kappa \ge Q\ge (\log N)^A.
$$

\bigskip

Throughout the paper, the implied constants in 
the symbols `$O$' and `$\ll$'
may occasionally, where obvious, depend on 
the  matrix $A$ and real positive parameters $\varepsilon$, and are absolute otherwise
(we recall that $U \ll V$  is equivalent
to $U = O(V)$).

For any sequence $\boldsymbol{\alpha}=(\alpha_k)_{k=1}^K$ of complex numbers, we write 
$$
\|\boldsymbol{\alpha} \|_\infty=\max_{k\le K}|\alpha_k|.
$$

 \section{Some Single and Double  Exponential Sums} 
 
 \subsection{Bounds on single sums}
 
We have the following bound, given by~\cite[Lemma~6]{MeShp21}, 
which is a generalisation of~\cite[Lemma~5.3]{AbSh}, which in turn 
improves and generalises the bound of~\cite[Theorem~1]{NiSh1}. 

\begin{lemma}
\label{lem:SingSum-s}  Assume that the characteristic polynomial of the matrix $A$ given by~\eqref{eq:MatrM} has two distinct roots in $\F_{p^2}$. 
 Let $t$ be the period of the sequence~\eqref{eq:Gen}.
For any integer numbers $N,K \geq 1$, $s\geq 2$ and $M \ge m_{s} > \ldots > m_1 \ge 1$, 
uniformly over $a_1, \ldots, a_s \in \F_p$ not all zeros, we have
$$
 \sum_{n=1}^{N} \ep\(  a_1  u_{m_1 n}+ \ldots + a_s  u_{m_s n}\) 
 \ll  sM\left(1+\frac{N}{t}\right)p^{1/2} \log p .
$$
\end{lemma}

We now immediately derive

\begin{cor}
\label{cor:SingSum} For any $\varepsilon > 0$ if the period $t$ of the sequence~\eqref{eq:Gen}
satisfies $t \ge p^{1/2+\varepsilon}$,  then for  arbitrary $N\ge p^{1/2+\varepsilon}$ 
and  $h \not \equiv 0 \pmod p$,  we have
$$
  S_{h}(N)   \ll N p^{-\varepsilon/2} .
$$
\end{cor}

\begin{proof} If $N \le t$ then by Lemma~\ref{lem:SingSum-s} we have 
$S_{h}(N)   \ll   p^{1/2}  \log p  \ll N p^{-\varepsilon/2}$ 

If $N > t$ then, using the periodicity of  the sequence~\eqref{eq:Gen},
we split the sum $S_{h}(N)$ into $O(N/t)$ sums $S_{h}(t)$ of length $t$ and one sum
$S_{h}(M)$
of length $M < t$. Using Lemma~\ref{lem:SingSum-s} we obtain $S_{h}(t) \ll  t p^{-\varepsilon/2}$
and also $S_{h}(M)  \ll  p^{1/2}  \log p \ll  t p^{-\varepsilon/2}$.
The result now follows.
\end{proof}

 \subsection{Bounds on double  sums}
We have the following bound which is essentially~\cite[Lemma~8]{MeShp21}. 
 We also formulate it in a slightly more precise form with 
$(\log p)^{1/2}$ as the proof actually gives  instead of  $p^{o(1)}$ 
as presented in~\cite[Lemma~8]{MeShp21}.

\begin{lemma}\label{lem:BilinExp}
 Assume that the characteristic polynomial of the matrix $A$ given by~\eqref{eq:MatrM} has two distinct roots in $\F_{p^2}$. 
 Let $t$ be the period of the sequence~\eqref{eq:Gen}.
For any    integers $M, K \ge 1$ and  any   sequences $\boldsymbol{\alpha}=(\alpha_k)_{k=1}^K$
and $\boldsymbol{\beta}=(\beta_m)_{m=1}^{M}$ of complex numbers with $\|\balpha \|_\infty,  \|\bbeta \|_\infty \le 1$, 
 uniformly over $h\in \F_p^*$,  we have
\begin{align*}
& \sum_{k=1}^{K} \sum_{m=1}^{M}
\alpha_k\,\beta_m \,\ep(h u_{k m}) \\
 & \qquad   \quad  \ll 
KM    \(M^{-1/2} + K^{-1/2}M^{1/2} p^{1/4} + M^{1/2} p^{1/4} t^{-1/2}    \)(\log p)^{1/2}.
\end{align*}
\end{lemma}

We now estimate double sums with variables limits of summation for
one variable.

\begin{lemma}\label{lem:DoubleSum2}
Let $K,M$ be positive integers and let $t$ be the period of the sequence~\eqref{eq:Gen}. 
Let $(L_m)$ and $(K_m)$ be sequences  
of nonnegative integer numbers  with $$
L_m<K_m< K.
$$
If the characteristic polynomial of the matrix $A$ given by~\eqref{eq:MatrM} has two distinct roots in $\F_{p^2}$, then
for any  sequences $\boldsymbol{\alpha}=(\alpha_k)_{k=1}^K$
and $\boldsymbol{\beta}=(\beta_m)_{m=1}^{M}$ of complex numbers with $\|\balpha \|_\infty,  \|\bbeta \|_\infty \le 1$, 
 uniformly over $h\in \F_p^*$,  we have
\begin{align*}
 \sum_{m =1}^{M}& \sum_{L_m<k\le K_m}
\alpha_k\,\beta_m \,\ep(h u_{k m})\\
&  
\ll
KM  \(M^{-1/2} + K^{-1/2}M^{1/2} p^{1/4} + M^{1/2} p^{1/4} t^{-1/2}    \)(\log p)^{1/2}\log K. 
\end{align*}
\end{lemma}

\begin{proof}
Write 
$$
S=\sum_{m =1}^{M} \sum_{L_m<k\le K_m}
\alpha_k\,\beta_m \,\ep(h u_{k m}).
$$
Using the orthogonality of exponential functions, for each inner sums we have
\begin{align*}
\sum_{L_m<k\le K_m} & \alpha_k \,\ep(h u_{k m})\\
&= \sum_{k\leq K}\alpha_k \sum_{L_m<s\leq K_m} \ep(h u_{s m})
\frac{1}{K} \sum_{-K/2 \leq r< K/2} \e_K(r(k-s))\\
&=\frac{1}{K} \sum_{-K/2 \leq r< K/2} 
\sum_{L_m<s\leq K_m} \e_K(-rs)\sum_{k\leq K}\alpha_k\ep(h u_{s m})\e_K(rk).
\end{align*}  
Here, for each $k\leq K$ and every integer $-K/2 \leq r< K/2$ we have by~\cite[Bound~(8.6)]{IwKow}, that
$$
\sum_{L_m<s\leq K_m} \e_K(-rs) \ll \frac{K}{r+1}.
$$
Let $\eta_{m,r}\ll 1$ be the complex number such that
$$
\sum_{L_m<s\leq K_m} \e_K(-rs) =\eta_{m,r} \frac{K}{r+1}.
$$
Thus
$$
S=
 \sum_{-K/2 \leq r< K/2}  \frac{1}{r+1}
\sum_{m =1}^{M} \sum_{k\le K}
\tilde{\alpha}_k\,\tilde{\beta}_m \,\ep(h u_{k m})
$$
with
$$
\tilde{\alpha}_k=\alpha_k \e_K(rk)
\mand 
\tilde{\beta}_m=\beta_m\eta_{m,r}  .
$$
As
$$
 \sum_{-K/2 \leq r< K/2}  \frac{1}{r+1}\ll \log K,
$$
Lemma~\ref{lem:BilinExp} yields
$$
S\ll
KM   
\(M^{-1/2} + K^{-1/2}M^{1/2} p^{1/4} + M^{1/2} p^{1/4} t^{-1/2}    \)(\log p)^{1/2}\log K, 
$$
which concludes the proof. 
\end{proof}

We also need the following bound on double exponential sums over certain 'hyperbolic'
regions. 

\begin{lemma}\label{lemma:hyper}
Let $H, K, M$ be positive integer numbers with $H<M$ and let $t$ be the period of the 
sequence~\eqref{eq:Gen}. Let $(L_m)$  be a sequences of nonnegative integer numbers. 
If the characteristic polynomial of the matrix $A$ given by~\eqref{eq:MatrM} has two 
distinct roots in $\F_{p^2}$, then
for any  sequences $\boldsymbol{\alpha}=(\alpha_k)_{k=1}^K$
and $\boldsymbol{\beta}=(\beta_m)_{m=1}^{M}$ of complex numbers with $\|\balpha \|_\infty,  \|\bbeta \|_\infty \le 1$, 
 uniformly over $h\in \F_p^*$,  we have
\begin{align*}
\sum_{m =H}^{M}  &\sum_{L_m<k\le K/m}
\alpha_k\,\beta_m \,\ep(h u_{k m})
\\
& 
\ll K\left(
H^{-1/2}
+
M K^{-1/2}p^{1/4}
+
M^{1/2}p^{1/4}t^{-1/2}
\right) (\log p)^{1/2} 
\log K. 
\end{align*} 
\end{lemma}

\begin{proof}
Let 
$$
I =  \fl{ \log H}   -1  \mand J = \rf{\log M}.
$$ 
By setting $\beta _m=0$ for $m<H$ and $m>M$, we have
\begin{align*}
\sum_{m =H}^{M}  \sum_{L_m<k\le K/m}&
\alpha_k\,\beta_m \,\ep(h u_{k m})
\\
&=\sum_{I \leq j\leq J  } 
\sum_{e^j< m \leq e^{j+1} } 
\sum_{L_m<k\le K/m}
\alpha_k\,\beta_m \,\ep(h u_{k m}).
\end{align*}
For each $j$, we use Lemma~\ref{lem:DoubleSum2} to derive 
\begin{align*}
    \sum_{m =H}^{M} & \sum_{L_m<k\le K/m}
\alpha_k\,\beta_m \,\ep(h u_{k m})
\\
&\ll
\sum_{I \leq j\leq J  } 
e^{j+1}\cdot \frac{K}{e^j}
\left(
e^{-(j+1)/2}+ \frac{e^{(j+1)/2} p^{1/4}}{(K/e^j)^{1/2}}+e^{(j+1)/2}p^{1/4}t^{-1/2}
\right)\\
&\qquad\qquad\qquad\qquad\qquad\qquad\qquad\qquad\qquad\qquad
(\log p)^{1/2}\log K\\
&\ll
\sum_{I \leq j\leq J  } 
K\left( e^{-j/2}+ e^j K^{-1/2} p^{1/4}+e^{j/2}p^{1/4}t^{-1/2}
\right)
(\log p)^{1/2}\log K\\
&\ll
K\left(
H^{-1/2}
+
M K^{-1/2}p^{1/4}
+
M^{1/2}p^{1/4}t^{-1/2}
\right)
(\log p)^{1/2} \log K, 
\end{align*}
which proves the result. 
\end{proof}

\section{Proof of Theorem~\ref{thm:SmoothSum3}}

\subsection{Combinatorial partition of the  sum} We set 
$$ 
L= Qp^{\varepsilon/4}.
$$
Clearly, we have  \begin{equation}\label{eq:L-bound}
Q\leq L\leq NQp^{-\varepsilon/8} 
\end{equation}
if $Q$ is large enough.

Let $p(s)$ denote the smallest prime divisor of an integer $s \ge 2$.

Following the idea of Vaughan~\cite[Lemma~10.1]{Vau}, we observe that
if  $n\in \cS(N,Q)$ with $n\geq L$, then $n$ can be written as
\begin{equation}\label{eq:rs}
n=r\cdot s, \quad \text{with } \qquad  L/Q\leq r < L, \quad  P(r)\leq p(s), \quad r  p(s)\geq L.
\end{equation}
One can have the representation~\eqref{eq:rs} by collecting  prime factors of $n$ into~$r$ starting from $p(n)$ and then using the rest of the prime factors in the increasing order, until we
reach $r \ge L/Q$.  We also see that $r <P(n) L/Q < L$.  Furthermore, we now choose $r$ to be the largest 
obtaining via the above procedure which still satisfies $r< L$. The maximality of $r$ implies that no 
remaining prime factors can be added to $r$ and thus $r p(s) \ge L$.  

We now associate with each $n$ a unique pair $(r,s)$ satisfying~\eqref{eq:rs}  and obtained via the 
above procedure.

Thus collecting the above pairs $(r,s)$ by the greatest prime factor $q$ of~$r$, we have 
\begin{equation}\label{T_h}
\begin{split} 
T_h(N,Q) &=   \sum_{\substack{L\leq n\leq N\\ P(n)\leq Q}}\ep\(hu_{n}\)+O(L)\\ 
&=
 \sum_{\substack{q\leq Q\\ q \text{ is prime}}} \ \sum_{\substack{L/Q\leq r\leq L\\ P(r)=q}} \ \sum_{\substack{N/L\leq s\leq N/r\\ p(s)\geq q\\ P(s)\leq Q}} \ep\(hu_{rs}\)+O(L). 
\end{split}
 \end{equation}
 
\subsection{Concluding the proof} 
Using the representation~\eqref{T_h} of the sum $T_h(N,Q)$, we apply Lemma~\ref{lemma:hyper} for every fixed $q$. 

To do so, we put $\balpha$ as the characteristic sequence of integers $s$ with $p(s)\geq q$ and $P(s)\leq Q$ and $\bbeta$ as the characteristic sequence of integers $r$ with $P(r)=q$. Then, for each fixed prime $q \le Q$,  Lemma~\ref{lemma:hyper} yields
\begin{align*}
 \sum_{\substack{L/Q\leq r\leq L\\ P(r)=q}} & \ \sum_{\substack{N/L\leq s\leq N/r\\ 
 q \le p(s) \le P(s)\leq Q}} \ep\(hu_{rs}\)\\ 
 & \ll N  \left(\frac{Q^{1/2}}{L^{1/2}}+ \frac{L  p^{1/4}}{N^{1/2}} + \frac{L^{1/2}p^{1/4}}{ t^{1/2}} \right)
 (\log p)^{1/2} \log N\\
   & \ll N  \(p^{-\varepsilon/8}+ \frac{Qp^{1/4+\varepsilon/4}}{N^{1/2}} + \frac{Q^{1/2}p^{1/4+\varepsilon/8}}{ t^{1/2}} \)(\log p)^{1/2} 
 \log N.
\end{align*} 
Taking the summation over primes $q\leq Q$, we get 
\begin{align*}
  T_h(N,Q)
    &\ll  NQ   \left(p^{-\varepsilon/8}
    + \frac{  Qp^{1/4+\varepsilon/4}}{N^{1/2}} + \frac{Q^{1/2}p^{1/4+\varepsilon/8}}{ t^{1/2}} \right)
  \frac{(\log p)^{1/2} \log N}{\log Q} +L\\
 &\ll NQ p^{-\varepsilon/8}  (\log p)^{1/2} \log N
\end{align*}
by~\eqref{eq:L-bound}. As $N\leq p^B$, we get
$$
T_h(N,Q)\ll N^{1- \varepsilon /(8B)}Q (\log p)^{3/2}
$$
and the result follows.

\section{Remarks}

Certainly  the most challenging open question in 
this area  is to obtain nontrivial 
results in the case of the period $t < p^{1/2}$. For such short 
periods no nontrivial results are known  even in the 
case of sums~\eqref{eq:Sing Sum} over consecutive integers.
In particular, methods of additive combinatorics, which stem from 
the groundbreaking result of Bourgain,    Glibichuk and  Konyagin~\cite{BGK}, 
do not apply to these sum.

Using a modification of the arguments of this paper, one can also 
study the distribution of elements of  the sequence~\eqref{eq:Gen}
at the moment of time satisfying various arithmetic conditions. 
However, it appears 
that studying sparse subsequences, as those with polynomial arguments, that is, $u_{f(n)}$,
$n =1,2, \ldots$, where $f(X) \in \Z[X]$, requires substantially new ideas.

It is certainly interesting to obtain analogues of
our results for 
orbits of polynomial dynamical system $x \mapsto F(x)$,
with a permutation polynomial $F \in \F_p[X]$. Unfortunately, 
due to the rapid  degree growth of the iterates $F^n$ 
even in the case of single sums over consecutive intervals
the saving against the trivial bound is at most logarithmic,
see~\cite{NiSh0,NiWi}. In turn this lead to rather weak bounds 
that cannot be applied to exponential sums over smooth number or primes. 

However, in the multidimensional case, several
polynomial systems $\cF = \{F_1, \ldots ,F_{m}\}$  of $m$ polynomials 
in $m$ variables over   $\F_p$ have been constructed (see~\cite{Ost1,OstShp1,OstShp2}), that  generate 
a permutation map   on the $\F_p^m$ and such that the degree
of its iterations grows polynomially.   So it is quite conceivable that one 
can obtain analogues of the results of this paper as well as 
of~\cite{MeShp21}
for the polynomial systems with slow degree growth of~\cite{Ost1,OstShp1,OstShp2}
as well as for special systems of~\cite{BhShp,Ost2,OstShp3}.

 \section*{Acknowledgement}

During the preparation of this work, L.M. was was partially supported by NRDI (National Research Development and Innovation Office, Hungary) grant FK 142960 and by the J{\'a}nos Bolyai Research Scholarship of the Hungarian Academyof Sciences and
I.E was supported in part by 
the  Australian Research Council  Grants DP230100530 and DP230100534.

\end{document}